\documentclass[12pt, a4paper,twoside]{article}
\pdfoutput=1
\usepackage{hbpaper,verbatim}
\usepackage{enumitem}
\usepackage{tikz-cd}
\usepackage{youngtab}
\usepackage{ytableau}
\usepackage{mathrsfs}
\usepackage{genyoungtabtikz}
\usepackage{tikz}
\usepackage{scalefnt}

\newcommand{\fqH}[1]{R^{\Lambda_{#1}}}   % finite quiver Hecke algebra

\newcommand{\df}{\operatorname{def}}
\newcommand{\proj}{\operatorname{proj}}
\newcommand{\spn}{\operatorname{span}}
\renewcommand{\End}{\operatorname{End}}
\newcommand{\Top}{\operatorname{Top}}
\newcommand{\wt}{{\rm wt}}
\newcommand{\Ht}{{\rm ht}}
\newcommand{\Rad}{{\rm Rad}}
  
\newcommand{\I}{I}   % index set
\newcommand{\rlQ}{\mathsf{Q}}   % root lattice
\newcommand{\wlP}{\mathsf{P}}   % weight lattice
\newcommand{\weyl}{\mathsf{W}}  % Weyl group
\newcommand{\cmA}{\mathsf{A}}  % Cartan matrix
\newcommand{\Par}{\mathcal{P}}   % partition
\DeclareMathOperator{\cont}{cont} % content of a partition
\newcommand{\ST}[1]{\mathsf{Std}(#1)}   % standard tableaux
\newcommand{\res}[1]{\mathrm{res}(#1)}   % residue sequence
\newcommand\tabupto[2]{#1_{\downarrow#2}}
\newcommand\shp[2]{\operatorname{Shp}(\tabupto{#1}{#2})}
\newcolumntype{M}[1]{>{\centering\arraybackslash}m{#1}}
\renewcommand{\mod}{{\rm \text{-}mod}}
\newcommand{\xik}[2]{\xi_{#1,#2}}

\title{Representation type of level $1$ KLR algebras in type $C$}
\author{Christopher Chung\\\normalsize Okinawa Institute of Science and Technology\\\normalsize Okinawa, Japan 904-0495 \\\texttt{\normalsize christopher.chung@oist.jp}\\[11pt]
Berta Hudak\\\normalsize Okinawa Institute of Science and Technology\\\normalsize Okinawa, Japan 904-0495 \\\texttt{\normalsize berta.hudak@oist.jp}\\[11pt]}
\date{}

\begin{document}

\renewcommand\auth{Chris Chung \& Berta Hudak}

\runninghead{Representation type of level $1$ KLR algebras in type $C$}


\toptitle 

\begin{abstract}
We determine the representation type for block algebras of the quiver Hecke algebras $\fqH{k}(\beta)$ of type $C^{(1)}_\ell$ for all $k$, generalising results of Ariki--Park for $\La = \La_0$.
\end{abstract}

\section{Introduction}

\emph{KLR algebras}, also called \emph{(affine) quiver Hecke algebras}, were introduced by Khovanov--Lauda \cite{kl09} and Rouquier \cite{rouq08} to give a categorification of the negative half of quantum groups. KLR algebras have natural finite-dimensional quotients $ R^\La(\beta) $ for a fixed dominant integral weight $ \La \in \wlP^+ $ and varying $ \beta \in \rlQ^+ $ a non-negative integral linear combination of simple roots, called the \emph{cyclotomic KLR algebras}, or \emph{cyclotomic quiver Hecke algebras}. The module categories of cyclotomic KLR algebras $ R^\La(\beta) $ for various $ \beta $, together with the induction and restriction functors between them give a categorification for the irreducible highest weight module $V(\La)$ over a quantum group. Brundan and Kleshchev \cite{bkisom} showed that the block algebras of the cyclotomic Hecke algebras are isomorphic to the cyclotomic KLR algebras of type $ A_{\ell}^{(1)} $, and so cyclotomic KLR algebras are a vast generalization of certain cyclotomic Hecke algebras, which are well understood.

In affine type $ A^{(1)}_{\ell} $, further advances have been made in understanding the structure of cyclotomic KLR algebras. In addition to the Brundan--Kleshchev isomorphism theorem, we also have a categorification theorem due to Ariki \cite{ariki96}, Brundan and Kleshchev which implies that in characteristic zero the simple objects in these categories correspond to the dual canonical basis of the integral highest weight module $V_{\bba} (\La)$ over the affine type $A$ Kac--Moody Lie algebra, using a Fock space construction.
Beyond type $A$, there is still much to discover about cyclotomic KLR algebras and their block algebras.
Classically, the representation type of block algebras of the Iwahori--Hecke algebra of the symmetric group was described by Erdmann and Nakano. % For a given quantum characteristic $\ell$, the Nakayama conjecture proved by James and Mathas tells us that the block algebras are parametrized by pairs of an $ \ell $-core partition and an $ \ell $-weight. The block algebras categorify weight spaces of the integrable $ U_q(\widehat{\mathfrak{sl}_\ell}) $-module $ V(\La_0) $ of highest weight $ \La = \La_0$, the zeroth fundamental weight. This can be viewed as the inspiration for the LLT conjecture by Lascoux, Leclerc and Thibon, and the starting point of Fock space theory. Fock spaces are certain infinite-dimensional representations for important algebras. Following work of Hayashi and Misra-Miwa \cite{MM90}, $ V(\La_0) $ can be identified as a Fock space with an action of the quantum affine algebra of type $ A^{(1)}_{\ell - 1} $ given in terms of the combinatorics of Young diagrams. Weights found in $ V(\La_0) $ have the form $ w \La_0 - k \delta $ where $ w \in W $ is an element of the Weyl group and $ k $ is a non-negative integer. Through the Misra-Miwa Fock space realization, we can view $ w \La_0 $ as the $ \ell $-core partition and $ k $ as the $ \ell $-weight. Through the Brundan-Kleshchev isomorphism, this led to a Lie theoretic approach to investigating the modular representation theory of the cyclotomic Hecke algebras, of which special cases include the finite Hecke algebras of types $B$ and $C$, and hence type $D$ as well via a Clifford embedding into type $B$.\\
Beyond type $A$, a general representation type classification in the style of Erdmann--Nakano for the block algebras of cyclotomic KLR algebras is a subject of active research.
For the special case $ \La = \La_0 $, this has been investigated in a series of recent papers by Ariki--Iijima--Park for type $ A^{(1)}_{\ell} $ and Ariki--Park for types  $ A_{2l}^{(2)}, D^{(2)}_{\ell + 1} $ and $ C^{(1)}_\ell $ \cite{APA1,APA2,apd,apc}, resulting in a Lie theoretic classification of representation type for $ R^{\La_0} (\beta) $ in the spirit of Erdmann--Nakano.
In affine types $ A $ (both twisted and untwisted) and $ D^{(2)} _{\ell+1} $, the representation types turned out to be governed by the weight, also called defect, introduced by Fayers \cite{fay06weights}, as a natural generalisation of the classification by Erdmann--Nakano.
In affine type $C$ Ariki--Park showed that $ R^{\La_0} (\delta) $ is no longer of finite representation type; in fact the algebra $ R^{\La_0} (\delta) $ is wild except in type $C^{(1)}_2$, where it is tame \cite{apc}.
Also, for arbitrary fundamental weight $ \La = \La_k $ (i.e.~level $1$) in untwisted affine ADE type, the set of maximal weights $ \max(\La) $ of $ V(\La) $ form a single Weyl group orbit \cite[Lem.~12.6]{kac}%, and so natural generalizations of $ \ell $-cores and $ \ell $-weights exist
.
The same is not true for affine type $ C $; as in this case $ \max(\La_k) $  consists of several Weyl group orbits, whose maximal representatives were given by Ariki and Park in \cite{apc} when $k=0$.
Nevertheless, we can show in this paper that the defect governs the representation type in affine type $C$ for arbitrary $ \La = \La_k$.
As a first step, we enumerate a set of maximal weight representatives 
\begin{equation*}
\begin{aligned} 
\xik{k}{i}&:= \alpha_{k+1} + 2\alpha_{k+2} + \dots + (i-1)\alpha_{k+i-1} + i( \alpha_{k+i} + \alpha_{k+i+1} + \dots + \alpha_{\ell-1})+ \frac{i}{2}\alpha_\ell; \\
\xik{k}{-i}&:= \alpha_{k-1} + 2\alpha_{k-2} + \dots + (i-1)\alpha_{k-i+1} + i( \alpha_{k-i} + \alpha_{k-i-1} + \dots + \alpha_{1})+ \frac{i}{2}\alpha_0 
\end{aligned}
\end{equation*}
for $ k\pm i \in I $. Thus, we need only investigate $ R^{\La_k}(m\delta - \xik{k}{\pm i}) $ for $ m \ge i/2 $, a fact we will also use to show that the defect is non-negative in level one.

We explain the cases where we need different arguments than \cite{apc}.
First, we consider $i = 0$ and $m=1$, that is the representation type of $\fqH{k}(\delta)$.
For the case $\ell=2$ and $k=1$, we give an explicit description of the indecomposable projective modules and prove that the algebra has tame representation type.
For larger $\ell$ and $k$, we are able to apply a recent interesting result of Ariki (\cite{arikirep}, see also \cref{trick}) which reduces the problem to finding two appropriate idempotents.
We show that if $\ell\neq2$ and $0 \le k \le \ell$, $\fqH{k}(\delta)$ is of wild representation type.
Then, we will show that $R^{\La_k}(\delta-\xik{k}{\pm 2})$ is of finite representation type, and hence that block algebras of defect one are Brauer tree algebras whose Brauer tree is a straight line with no exceptional vertex. We will also see that depending on $k$, there can be two inequivalent Morita equivalence classes of blocks of defect one, with distinct number of simple modules $k+1$ or $\ell-k+1$ respectively.
Next, we deal with the representation type of $\fqH{k}(2\delta-\xik{k}{\pm4})$ and using \cref{trick} again, we arrive at the result that this algebra has wild representation type as well.
Finally, we handle the remaining cases by using the same arguments as in \cite{apc} with some slight modifications.
Our results are summarised in \cref{main}. \\

\textbf{Acknowledgements.} Both authors thank Professor Susumu Ariki and Professor Liron Speyer for their helpful comments and inspiring correspondence. The second author would also like to thank Professor Susumu Ariki for hosting her at Osaka University during which stay much of the research for this project was carried out.

\section{Background}

\subsection{Lie theory notation}\label{lie}

We mostly follow \cite{kac} and use standard notation for the root datum.

Let $\ell\in \{2,3,\dots\}$ and $\I= \{ 0,1,2, \ldots, \ell \} $.

The affine Cartan matrix of type $C^{(1)}_\ell$ is given by
\[
\cmA = (a_{ij})_{i,j\in I} = \begin{pmatrix}
2 & -1 & 0 & \cdots & 0 & 0 & 0\\
-2 & 2 & -1 & \cdots & 0 & 0 & 0\\
0 & -1 & 2 & \cdots & 0 & 0 & 0\\
\vdots & \vdots & \vdots & \ddots & \vdots & \vdots & \vdots\\
0 & 0 & 0 & \cdots & 2 & -1 & 0\\
0 & 0 & 0 & \cdots & -1 & 2 & -2\\
0 & 0 & 0 & \cdots &  0 & -1 & 2
\end{pmatrix}.
\]

We have \emph{simple roots} $\Pi=\{\alpha_i \mid i\in \I \}$ and \emph{fundamental weights} $\{\Lambda_i \mid i\in \I \}$ in the \emph{weight lattice} $\wlP$, and simple coroots $\Pi^\vee=\{\alpha_i^\vee \mid i\in \I\}$ in the \emph{dual weight lattice} $\wlP^\vee$. 
Each $\al_i$ gives rise to a $\bbz$-linear transformation $r_i$ acting on $\wlP$ by
$r_i\Lambda=\Lambda-\langle h_i, \Lambda\rangle\alpha_i$, for $\Lambda\in \wlP$. Let $d \in P^\vee$ be the element such that $ \alpha_i(d) = \delta_{0,i}$.
We denote by $\weyl$ the Weyl group, the group generated by $\{r_i \mid i\in \I \}$.

There is a $\weyl$-invariant symmetric bilinear form $( - , - )$ on $\wlP$ satisfying the following:
\begin{enumerate}
    \item $( \Lambda_i, \alpha_j ) = d_j \delta_{ij}$; and 
    \item $( \alpha_i, \alpha_j ) = d_i a_{ij}$ 
where $d = (2,1\dots,1,2)$ if $\ell<\infty$.\label{ds}
\end{enumerate}
We denote the set of \emph{dominant integral weights} by
\[
\wlP^+ = \{\Lambda \in \wlP \mid \langle\alpha_i^\vee,\Lambda\rangle \ge 0 \text{ for all } i\in \I\},
\]
where $\langle$ , $\rangle$ is the natural pairing.
We call $\rlQ:=\bigoplus_{i\in I} \bbz \alpha_i$ the \emph{root lattice} and $\rlQ^+ = \bigoplus_{i\in I} \bbz_{\ge 0} \alpha_i$ the \emph{positive cone} of the root lattice.
For $C^{(1)}_\ell$, we have the \emph{null root} given by
\[
\delta = \alpha_0 + 2\alpha_1 +\dots + 2\alpha_{\ell-1} +\alpha_\ell.
\]
The \emph{defect} of $\beta\in\rlQ^+$ (relative to $\La$) is given by
\[
\df_\La(\beta) = (\La,\beta)-\frac{1}{2}(\beta,\beta).
\]
When it is clear from context, we will omit $\La$ from the subscript and write $ \df(\beta) $ instead of $ \df_\La (\beta) $.

We set $\xik{k}{0} := 0$, and for $k + i \in \I$ or $k - i \in \I$ we set
\begin{equation} \label{xik}
\begin{aligned} 
\xik{k}{i}&:= \alpha_{k+1} + 2\alpha_{k+2} + \dots + (i-1)\alpha_{k+i-1} + i( \alpha_{k+i} + \alpha_{k+i+1} + \dots + \alpha_{\ell-1})+ \frac{i}{2}\alpha_\ell; \\ %\text{ and}
\xik{k}{-i}&:= \alpha_{k-1} + 2\alpha_{k-2} + \dots + (i-1)\alpha_{k-i+1} + i( \alpha_{k-i} + \alpha_{k-i-1} + \dots + \alpha_{1})+ \frac{i}{2}\alpha_0 
\end{aligned}
\end{equation}
respectively.

Note that if $i\ne0$ then
\begin{equation}\label{xikev}
\xik{k}{\pm i}(\al_j^\vee) = \left\{
                              \begin{array}{ll}
                                -1 & \hbox{ if } j=k, \\
                                1 & \hbox{ if } j=k\pm i, \\
                                0 & \hbox{ otherwise,}
                              \end{array}
                            \right. 
\end{equation}
and they form a basis for $\sum_{i\in \I \setminus \{k\}} \bbq \alpha_i$.

\begin{lem}\label{lem:defcalc} For $ m \ge i/2 $ and so $m\delta - \xik{k}{\pm i} \in \rlQ^+ $, we have that
   $ \df(m\delta - \xik{k}{\pm i}) = 2m - \frac{1}{2} i $ (relative to $\La_k$).
\end{lem}

\begin{proof}
We have the following calculation:
\begin{equation*}
    \begin{aligned}
    \df_{\La_k}(m\delta - \xik{k}{\pm i}) &= (\La_k, m\delta - \xik{k}{\pm i}) - \frac{1}{2} (m\delta - \xik{k}{\pm i},m\delta - \xik{k}{\pm i}) \\
    &= (\La_k, m\delta) - \frac{1}{2} (\xik{k}{\pm i},\xik{k}{\pm i}) \\
    &= 2m - \frac{1}{2} i
\end{aligned}
\end{equation*}
as required.
\end{proof}

Let $\fkg = \fkg(\cmA) $ be the affine Kac--Moody algebra associated with the Cartan datum $(\cmA, \wlP, \Pi, \Pi^{\vee})$ and let $U_q(\fkg)$ be its quantum group.
The quantum group $U_q(\fkg)$ is a $\bbc(q)$-algebra generated by $f_i$, $e_i$ $(i\in \I)$ and $q^h$ $(h\in \wlP^\vee)$ with certain relations (more details can be found in \cite[Chap.\ 3]{HK02}).
Let $\bba=\bbz[q,q^{-1}]$ and denote by $U_\bba^-(\fkg)$ the subalgebra of $U_q(\fkg)$ generated by $f_i^{(n)} := f_i^n / [n]_i!$ for $i\in \I$ and $n\in \bbz_{\ge0}$,
where  $q_i = q^{d_i}$ and
\begin{equation*}
 \begin{aligned}
 \ &[n]_i =\frac{ q^n_{i} - q^{-n}_{i} }{ q_{i} - q^{-1}_{i} },
 \ &[n]_i! = \prod^{n}_{k=1} [k]_i.
 \end{aligned}
\end{equation*}

For some $\Lambda \in \wlP^+$, let $V(\Lambda)$ be the irreducible highest weight $U_q(\fkg)$-module with highest weight $\Lambda$
and $V_\bba(\Lambda)$ the $U_\bba^-(\fkg)$-submodule of $V(\Lambda)$ generated by the highest weight vector.

The Fock space representation for $U_q(C_\ell^{(1)})$ was first constructed in \cite{KMM93} by folding the Fock space representation for $U_q(A_{2\ell-1}^{(1)})$ via the Dynkin diagram automorphism (see \cref{res}).
Later, the combinatorics of the Fock space and its crystal base were described in terms of tableaux and Young diagrams (\cite{KimShin04, prem04}).
The next section focuses on explaining this combinatorial realisation for $V(\La_k)$.

\subsection{Partitions and tableaux}

\begin{defn}
For $n\ge 0$, a \emph{partition} of $n$ is a weakly decreasing sequence of non-negative integers $\la = (\la_1, \la_2, \dots)$ such that the sum 
$|\la|=\la_1+\la_2+\cdots$ is equal to $n$. If $\la$ is a partition of $n$ we write $\la\vdash n$.
We write $\varnothing$ for the unique partition of 0.
We will denote the set of partitions of $n$ by $\Par_n$.
\end{defn}

For any $\la\in \Par_n$, we define its \emph{Young diagram} $[\la]$ to be the set
\[
\{(r,c) \in \bbz_{\ge 1} \times \bbz_{\ge 1}  \mid c\le \la_r\}.
\]

Note that we will depict a Young diagram of a partition using the English convention (i.e.\ successive rows of the diagram are lower down the page).

We define $f_\ell : \bbz \rightarrow I$ by $k \mapsto |k|$ if $\ell = \infty$ and, if $\ell \ne \infty$, $f_\ell : \bbz / 2\ell\bbz \rightarrow I $ by
\begin{equation}\label{res}
\begin{aligned}
&f_\ell(0 + 2\ell\bbz) = 0, \quad f_\ell(\ell+2\ell\bbz) = \ell, \\
 &f_\ell(k + 2\ell\bbz) = f_\ell(2\ell-k + 2\ell\bbz) = k \quad  \text{ for } 1 \le k \le \ell-1.
\end{aligned}
\end{equation}
Let $p$ be the natural projection $\bbz \rightarrow \bbz/ 2\ell\bbz$ 
%if $\ell \ne \infty$ and $p = \textrm{id} $ if $\ell=\infty$
.
We set
$\pi_\ell=f_\ell \circ p: \bbz \rightarrow I$.
If there is no confusion, we will denote $\pi_\ell(k)$ by $\overline{k}$, for $k\in \bbz$.

Given a \emph{charge} $\kappa\in \bbz$, we define $\Lambda_\kappa\in\wlP^+$ by $\Lambda_\kappa:= \Lambda_{\overline{\kappa}}.$
If $\la$ is a partition of $n$, then to any node $A = (r,c) \in [\la]$ we can associate its \emph{residue} defined by
\[
\res A = \overline{\kappa+c-r}.
\]
If $\res A = i$, we call $A$ an $i$-node.

\begin{eg}
If $\la=(8,6,6,5,2)$, $\kappa=2$ and $\ell=4$, then $\la$ has the following residue pattern.
\[
\young(23432101,123432,012343,10123,21)
\]
\end{eg}

We say that a node $A\in[\la]$ is \emph{removable} (resp.~\emph{addable}) if $[\la]\setminus A$ (resp.~$[\la]\cup A$) is a valid Young diagram for a partition of $n-1$ (resp.~$n+1$).
We write $\la\nearrow A$ (resp.~$\la\swarrow A$) as shorthand for the partition whose Young diagram is $[\la]\setminus A$ (resp.~$[\la]\cup A$).
For an $i$-node $A\in[\la]$, we set
\begin{align*}
N_0(\la) &= \#\{\text{$0$-coloured boxes in $\la$}\}, \\
d_i(\la) &= \#\{\text{addable $i$-nodes of $[\la]$} \} - \#\{\text{removable $i$-nodes of $[\la]$}\}, \\
d_A(\la) &= d_i \cdot (\#\{\text{addable $i$-nodes of $[\la]$ below } A\} - \#\{\text{removable $i$-nodes of $[\la]$ below } A\}), \\
d^A(\la) &= d_i \cdot (\#\{\text{addable $i$-nodes of $[\la]$ above } A\} - \#\{\text{removable $i$-nodes of $[\la]$ above } A\})
\end{align*}
where $d_i$ is given in \cref{lie}. We define the \emph{Fock space $\mathcal{F}(\kappa)$ with charge $\kappa$} to be the $\bbq(q)$-vector space  with basis consisting of partitions of $n$. 
For a Young diagram $[\la]$, $\mathcal{F}(\kappa)$ has a $U_q(\mathfrak{g}(\cmA))$-module structure defined by 
\begin{align*}
% q^d is the degree operator
q^d \la = q^{-N_0(\la)} \la, \qquad&\qquad   e_i \la = \sum_{A} q^{d_A(\la)} \la\nearrow A, \\
q^{\al_i^\vee} \la = q^{d_i(\la)}\la, \qquad&\qquad  f_i \la = \sum_{A} q^{-d^A(\la)} \la\swarrow A,
\end{align*}
where $A$ runs over all removable $i$-nodes and all addable $i$-nodes respectively \cite[Theorem~2.3]{prem04}).

We identify the basis of $\mathcal{F}(\kappa)$ with the set of all Young diagrams.
Its crystal structure can be described by considering the usual {\it $i$-signature}.
For a Young diagram $[\lambda]$, we consider all addable or removable $i$-nodes $a_1, a_2, \ldots, a_m$ of $[\lambda]$ from top to bottom and to each $a_j$ of $[\lambda]$, we assign its signature $s_j$ as $+$ (resp.\ $-$) if it is addable (resp.\ removable).
We cancel out all possible $(-,+)$ pairs in the $i$-signature $(s_1, \ldots, s_m  )$ to obtain the \emph{reduced $i$-signature}, which is a sequence of $+$'s is followed by $-$'s. We call the removable node corresponding to the leftmost $-$ in the $i$-signature a \emph{good node} and the addable node corresponding to the rightmost $+$ in the
$i$-signature a \emph{cogood node}.
We define $\tilde{f}_i \lambda$ to be the Young diagram obtained from $[\lambda]$ by adding a box at the cogood node.
Similarly, we define $\tilde{e}_i \lambda$ to be a Young diagram obtained from $[\lambda]$ by removing the box at the good node. The operators $\tilde{e}_i$ and $\tilde{f}_i$ defined above coincide with Kashiwara's operators cf. \cite[Theorem~3.3]{prem04}.
Here our choice of convention is compatible with that of \cite[\S3.6]{BK09} which deals with the type A Fock space, and differs from \cite{KimShin04, prem04} in the choice of convention used when reducing the $i$-signature. %by flipping the Young diagrams diagonally. 

Then (see \cite[Section~2]{prem04}), the $U_q(\mathfrak{g})$-submodule of $\mathcal{F}(\kappa)$ with $ \La_\kappa = \La_k $ generated by the empty partition $ \varnothing_k $ is isomorphic to the irreducible integrable highest weight module $V(\La_k)$, and the crystal graph for the $U_q(\mathfrak{g})$-crystal $V(\La_k)$ is the directed coloured graph with vertices the set of partitions that can be obtained from repeated application of the operators $\tilde{f}_i$, $i \in I $ to $\varnothing_k$ and $i$-coloured edges $ \la \xrightarrow{i} \mu $ whenever $\mu = \tilde{f}_i \la $ (or equivalently, $\tilde{e}_i \mu = \la $). We will call a partition \emph{Kleshchev} if it is a vertex in the crystal graph of $V(\La_k)\subset\mathcal{F}(\kappa)$ cf. \cite[\S6F]{EM22}. By Theorem C of \cite{EM22}, partitions of $n$ that are Kleshchev form a complete set of labels for the set of simple $R^\La_n$-modules. 

%In \cite[Thm.5]{KimShin04}, a criterion for evaluating if a given partition is Kleshchev is given, but it is not as straightforward as in type $A$. \\

\begin{defn}
Let $\la \in \Par_n$. A \emph{$\la$-tableau} is a bijection $\ttt :[\la] \rightarrow \{1,\dots,n\}$.
We depict $\ttt$ by filling each node $(r,c)\in [\la]$ with $\ttt(r,c)$.
We call a tableau $\ttt$ \emph{standard} if the entries increase along rows and down the columns.
We denote the set of standard tableaux by $\ST\la$.
We will denote by $\ttt^\la$ the unique $\la$-tableau where nodes are filled by $1,2,\dots,n$ along the successive rows and call $\ttt^\la$ the \emph{initial tableau}.
\end{defn}

For each $\la$-tableau $\ttt$, we have the associated residue sequence
\begin{align*}
\bfi^\ttt &= (i_1,i_2,\dots,i_n)\\
&=(  \res { \ttt^{-1}(1) } ,\   \res { \ttt^{-1}(2)} , \dots,  \res { \ttt^{-1}(n)} ).
\end{align*}
We will write $\bfi^\la$ for $\bfi^{\ttt^\la}$.

\begin{eg}
If $\la=(4,3,3,2)$, $\kappa=1$ and $\ell=3$, then
\[
\ttt^\la=\young(1234,567,89<10>,<11><12>)
\]
and we have that $\bfi^\la=(123201210121)$.
\end{eg}

Let $\ttt$ be a $\la$-tableau and choose some $0\le m\le n$.
We denote by $\tabupto\ttt m$ the set of nodes of $[\la]$ whose entries are less than or equal to $m$.
If $\ttt\in \ST\la$, then $\tabupto\ttt m$ is a standard tableau for some partition, which we call $\shp\ttt m$.

For any $\la\in\Par_n$ and $\ttt\in\ST\la$ we define the \emph{degree} $\deg\ttt$ of $\ttt$ as follows.
If $n=0$ then $\ttt$ is the unique $\varnothing$-tableau and we set $\deg\ttt:=0$.
Otherwise, let $A=\ttt^{-1}(n)\in[\la]$ and suppose $A$ is an $i$-node. We set inductively
\begin{align*} %\label{deg}
\deg\ttt := \deg\tabupto\ttt{n-1} + d_A(\la).
\end{align*}

\begin{eg}
If $\la=(3,3,1,1)$, $\kappa=1$ and $\ell=2$, then $\la$ has the following residue pattern
\[
\young(121,012,1,2)
\]
and if $\ttt$ is the tableau
\[
\young(134,258,6,7)
\]
then
\[
\deg\ttt=0+0+0+(1+1)+1+0+0-2=1.
\]
\end{eg}

\begin{defn}
We define the \emph{content} of the partition $\la$ to be
\[
\cont(\la) = \sum_{A\in[\la]} \alpha_{\res A}\in \rlQ^+.
\]
The \emph{defect} %, or also called the \emph{weight}, 
of a partition $\la$ is
\[
\df(\la):= \df(\cont(\la)).
\]
\end{defn}

Recall from \cite[\S~12.6]{kac} that a weight $\mu$ of $V(\Lambda)$ is {\it maximal} if $\mu + \delta$ is not a weight of $V(\Lambda)$. Let $\max(\Lambda)$ be the set of all maximal weights of $V(\Lambda)$. The following is a generalisation of Ariki and Park's result for $ \La = \La_0 $ to $ \La = \La_k $:

\begin{prop} \label{prop:maximal}
For the weight system of the $\fkg(\cmA)$-module $V(\Lambda_k)$ in type $C_\ell^{(1)}$, we have
\begin{enumerate}
\item $ \max(\Lambda_k) \cap \wlP^+ = \{  \Lambda_k + \xik{k}{\pm i} - \frac{i}{2}\delta \mid k\pm i\in I,\ \text{$i$ is an even nonnegative integer} \},$ and
\item $\mu$ is a weight of $V(\Lambda_k)$ if and only if $\mu = w \eta - m \delta$ for some $w\in \weyl$, $\eta \in \max(\Lambda_k) \cap \wlP^+ $ and $m\in \bbz_{\ge 0}$.
\end{enumerate}
\end{prop}

\begin{proof}
The details here are similar to \cite[Proof of Prop.~5.1]{apc}: Let $\mu \in \max(\Lambda_0) \cap \wlP^+$. 
Since $ \{ \xik{k}{\pm i} \mid k\pm i \in I \} $ satisfies \cref{xikev} and $\{ \xik{k}{\pm i} \mid i\neq 0 \text{ and } k\pm i \in I \}$ forms a basis of $\sum_{i\in I\setminus \{k\}} \bbq \alpha_i$, by the same computations $\mu = \La_k + \xik{k}{\pm i} + t \delta$ for some $i\neq 0$ such that $ k\pm i \in I $ or $\mu = \La_k+t\delta$. In the latter case, $\mu$ maximal implies that $\mu = \La_k = \La_k+\xik{k}{0}$. 
In the former case, $\La_k-\mu \in\rlQ^+$ and \eqref{xik} together imply that $i$ is even.

Then, to show that $t = -\frac{i}{2}$, we need to consider the cases $\mu = \La_k+\xik{k}{-i}+t\delta$ and $\mu = \La_k + \xik{k}{i} + t \delta$ separately.
For the case $\mu = \La_k+\xik{k}{-i}+t\delta$, using the partition
$$ \lambda(-i) = (\underbrace{\ell-k+i/2,\ell-k+i/2,\ldots, \ell-k+i/2}_{i}) $$
in the Fock space $\mathcal{F}(\kappa)$ with $\La_\kappa=\La_k$ and considering its residue pattern we have the corresponding weight 
\begin{align*}
    %\wt(\lambda(-i)) &= 
    &\Lambda_k - \left( \frac{i}{2}\alpha_\ell + i (\alpha_{\ell-1} +\cdots+ \alpha_k) + (i-1)\alpha_{k-1} + (i-2)\alpha_{k-2} \cdots + \alpha_{k-i+1} \right) \\
        &= \La_k + \xik{k}{-i} - \frac{i}{2}\delta,
\end{align*}
and so by \cref{thm:dim} in the next section that $\dim R^{\La_k} (\frac{i}{2}\delta - \xik{k}{-i}) \neq 0 $ and so $\La_k + \xik{k}{-i} - \frac{i}{2}\delta$ is a weight of $V(\La_k)$. Furthermore, $\La_k + \xik{k}{-i} - \frac{i}{2}\delta$ is maximal since $ (-\xik{k}{-i} + \frac{i}{2}\delta) - \delta \notin \rlQ^+ $.

For the case $\mu = \La_k+\xik{k}{+i}+t\delta$, we use the partition
$$ \lambda(+i) = (\underbrace{i,i,\ldots, i}_{k+i/2}) $$
in the Fock space $\mathcal{F}(\kappa)$ with $\La_\kappa=\La_k$ instead, which has corresponding weight
\begin{align*}
    %\wt(\lambda(+i)) &= 
    &\Lambda_k - \left( \frac{i}{2}\alpha_0 + i (\alpha_{1} +\cdots+ \alpha_k) + (i-1)\alpha_{k+1} + (i-2)\alpha_{k+2} \cdots + \alpha_{k+i-1} \right) \\
        &= \La_k + \xik{k}{+i} - \frac{i}{2}\delta,
\end{align*}
and so $\La_k + \xik{k}{+i} - \frac{i}{2}\delta$ is a weight of $V(\La_k)$; it is maximal since $ (-\xik{k}{+i} + \frac{i}{2}\delta) - \delta \notin \rlQ^+ $.

The second part is also argued similarly: $ \max(\La_k) $ is $\weyl$-invariant by \cite[Prop.~10.1]{kac} and furthermore $ \max(\Lambda_k) = \weyl (\max(\Lambda_k) \cap \wlP^+) $ by \cite[Cor.~10.1]{kac}. Thus, given any weight $\mu$ of $V(\Lambda_k)$, by \cite[(12.6.1)]{kac} there exists a unique $ \zeta \in \max(\Lambda_k)$ and a unique $m \in \bbz_{\ge0}$ such that $\mu = \zeta - m\delta$.
\end{proof}

\subsection{Quiver Hecke algebras}
Let $\bbf$ be an algebraically closed field and $(\cmA, \wlP, \Pi, \Pi^{\vee})$ the Cartan datum from \cref{lie}. We set polynomials $\calq_{i,j}(u,v)\in\bbf[u,v]$, for $i,j\in I$,
of the form
\begin{align*}
\mathcal{Q}_{i,j}(u,v) = \left\{
                 \begin{array}{ll}
                   \sum_{p(\alpha_i|\alpha_i)+q (\alpha_j|\alpha_j) + 2(\alpha_i|\alpha_j)=0} t_{i,j;p,q} u^pv^q & \hbox{if } i \ne j,\\
                   0 & \hbox{if } i=j,
                 \end{array}
               \right.
\end{align*}
where $t_{i,j;p,q} \in \bbf$ are such that $t_{i,j;-a_{ij},0} \ne 0$ and $\mathcal{Q}_{i,j}(u,v) = \mathcal{Q}_{j,i}(v,u)$.
The symmetric group $\fkS_n = \langle s_k \mid k=1, \ldots, n-1 \rangle$ acts on $I^n$ by place permutations.

\begin{defn} 
The {\it cyclotomic quiver Hecke algebra} $R^\La_n$ associated with polynomials $(\calq_{i,j}(u,v))_{i,j\in I}$ and $\La_k\in \wlP$
is the $\bbz$-graded $\bbf$-algebra generated by three sets of generators
$$\{e(\nu) \mid \nu = (\nu_1,\ldots, \nu_n) \in I^n\}, \;\{x_r \mid 1 \le r \le n\}, \;\{\psi_j \mid 1 \le j \le n-1\} $$
subject to the following list of relations:
{\allowdisplaybreaks
\begin{align*}
e(\nu)e(\nu')&=\delta_{\nu,\nu'} e(\nu); \\
\sum_{\nu \in \I^n} e(\nu)&=1;\\
x_re(\nu)&=e(\nu)x_r;\\
\psi_r e(\nu) &= e(s_r \nu) \psi_r;\\
x_rx_s&=x_sx_r;\\
\psi_rx_s&=\mathrlap{x_s\psi_r}\hphantom{\smash{\begin{cases}
\frac{Q_{\nu_r,\nu_{r+1}}(x_r,x_{r+1}) - Q_{\nu_r,\nu_{r+1}}(x_{r+2},x_{r+1})}{x_r - x_{r+2}} e(\nu)\end{cases}}}\kern\nulldelimiterspace
\text{if } s\neq r,r+1;\\
\psi_r\psi_s&=\mathrlap{\psi_s\psi_r}\hphantom{\smash{\begin{cases}
\frac{Q_{\nu_r,\nu_{r+1}}(x_r,x_{r+1}) - Q_{\nu_r,\nu_{r+1}}(x_{r+2},x_{r+1})}{x_r - x_{r+2}} e(\nu)\end{cases}}}\kern\nulldelimiterspace\text{if } |r-s|>1;\\
x_r \psi_r e(\nu) &=(\psi_r x_{r+1} - \delta_{\nu_r,\nu_{r+1}})e(\nu);\\
x_{r+1} \psi_r e(\nu) &=(\psi_r x_r + \delta_{\nu_r,\nu_{r+1}})e(\nu);\\
\psi_r^2 e(\nu)&=Q_{\nu_r,\nu_{r+1}}(x_r,x_{r+1})e(\nu);\\
(\psi_{r+1}\psi_{r}\psi_{r+1} - \psi_{r}\psi_{r+1}\psi_{r})e(\nu)&=\begin{cases}
\frac{Q_{\nu_r,\nu_{r+1}}(x_r,x_{r+1}) - Q_{\nu_r,\nu_{r+1}}(x_{r+2},x_{r+1})}{x_r - x_{r+2}} e(\nu) & \text{if } \nu_r=\nu_{r+2},\\
0 & \text{otherwise};\end{cases}
\end{align*}
}for all admissible $r,s,\nu,\nu'$, and $x_1^{\langle \alpha^\vee_{\nu_1},\Lambda \rangle} e(\nu)=0$ for $\nu \in I^n$.
\end{defn}

The algebra $R^\La_n$ is $\bbz$-graded by setting
\[
\deg(e(\nu))=0, \qquad \deg(x_r e(\nu)) = (\alpha_{\nu_r},\alpha_{\nu_r}), \qquad \deg(\psi_s e(\nu)) = -( \alpha_{\nu_s}, \alpha_{\nu_{s+1}})
\]
for all admissible $r,s$ and $\nu$.

For some $\beta \in \rlQ^+$ with $\Ht(\beta) = n$, we set
\[
\I^\beta = \{\nu \in \I^n \mid \alpha_{\nu_1} + \dots + \alpha_{\nu_n} = \beta\}.
\]
Then $e(\beta):=\sum_{\nu\in I^\beta} e(\nu)$ is a central idempotent. We define $R^\Lambda(\beta):=R^\Lambda_n e(\beta)$, which is also an $\bbf$-algebra. 
It is clear that $R^\Lambda(\beta)$ may be defined by the same set of generators and relations if we replace $I^n$ with $I^\beta$ in both. 
The cyclotomic quiver Hecke algebra $R^\La_n$ can be decomposed into a direct sum of $\bbf$-algebras:
\[
R^\La_n=\bigoplus_{\substack{\beta\in \rlQ^+\\ \Ht(\beta) = n}} R^\Lambda(\beta),
\]
(see \cite[\S~2.1]{aps}) and we refer to $R^\Lambda(\beta)$ as a \emph{block algebra} of $R^\La_n$. In general, it is not known that $R^\Lambda(\beta)$ is indecomposable, but as we show later this is the case in level one when $\beta$ has defect zero and one (\cref{prop:defect0,cor:defect1} below). 
The block algebras for $\La = \La_0$ were referred to as \emph{finite quiver Hecke algebras} in \cite{apc}. If we drop the relation $x_1^{\langle \alpha^\vee_{\nu_1},\Lambda \rangle} e(\nu)=0$ for $\nu \in I^\beta$, we obtain the \emph{quiver Hecke algebra} $R(\beta)$.

We define the \emph{defect} of a nonzero block algebra $R^\Lambda(\beta)$ to be the defect of $\beta \in \rlQ^+$. 

\begin{lem} \label{lem:def}
    The defect of a nonzero block algebra $R^{\La_k}(\beta)$ is non-negative.
\end{lem}

\begin{proof}
    By the categorification theorem, $R^\Lambda(\beta) \neq 0$ precisely when $ \La - \beta $ is a weight of $ V(\La) $. By \cref{prop:maximal}, when $ \La = \La_k $ this means that $\La - \beta = w \La + w \xik{k}{\pm i} - m \delta$ for $ k\pm i \in I $, $ m \ge i/2 $ and $w \in \weyl$ and so $ \beta = \La - w \La - w \xik{k}{\pm i} + m \delta$.
    By direct computation for $\beta \in \rlQ^+$, $\La - w^{-1} \La + w^{-1}\beta$ and $\beta$ have the same defect and so the defect of $R^{\La_k}(\beta)$ is $\df(\beta)=\df(m\delta - \xik{k}{\pm i}) = 2m - \frac{i}{2} \ge \frac{i}{2} \ge 0 $ by \cref{lem:defcalc}.
\end{proof}

This is true in higher levels as well, details will appear in an upcoming paper \cite{CMS23}.

In the rest of this section, we recall some important results which will be used in our proofs. Recall that by direct computation, $R^\Lambda(\beta)$ and $R^\Lambda(\Lambda-w\Lambda+w\beta)$ for $w \in \weyl$ have the same defect. In fact, a much stronger statement holds:
\begin{prop} [cf. {\cite[Cor.~4.8]{APA2}}] \label{reptype}
For $w \in \weyl$, $R^\Lambda(\beta)$ and $R^\Lambda(\Lambda-w\Lambda+w\beta)$ are derived equivalent; furthermore, they have the same number of simple modules and the same representation type.
\end{prop}

We denote the direct sum of the split Grothendieck groups of the categories $R^\Lambda (\beta)\text{-}\proj$ of finitely generated projective graded $R^\Lambda(\beta)$-modules by
\[
K_0(R^\Lambda) = \bigoplus_{\beta\in \rlQ^+} K_0(R^\Lambda (\beta)\text{-}\proj). 
\]
Note that $K_0(R^\Lambda)$ has a free $\bba$-module structure
induced from the $\bbz$-grading on $R^\Lambda(\beta)$, i.e.\ $(qM)_k =
M_{k-1}$ for a graded module $M = \bigoplus_{k \in \bbz} M_k $. Let
$e(\nu, i)$ be the idempotent corresponding to the concatenation of
$\nu$ and $(i)$, and set $e(\beta, i) = \sum_{\nu \in I^\beta}
e(\nu, i)$ for $\beta \in \rlQ^+$. Then we define the induction functor $F_i: R^\Lambda(\beta)\mod \rightarrow R^\Lambda(\beta+\alpha_i)\mod$ and
the restriction functor $E_i: R^\Lambda(\beta+\alpha_i)\mod \rightarrow R^\Lambda(\beta)\mod$ by
$$ F_i(M) = R^\Lambda(\beta+ \alpha_i) e(\beta,i) \otimes_{R^\Lambda(\beta)}M, \qquad \ E_i(N)  =  e(\beta,i)N ,$$
for an $R^\Lambda(\beta)$-module $M$ and an $R^\Lambda(\beta+ \alpha_i)$-module $N$.

\begin{thm} [{\cite[Thm.\ 5.2]{kk12}}] \label{endo}
Let $l_i = \langle h_i, \Lambda - \beta  \rangle$, for $i\in I$. Then one of the following  isomorphisms of endofunctors
on $R^\Lambda(\beta)\mod$ holds.
\begin{enumerate}
\item If $l_i \ge 0$, then
$$ q_i^{-2}F_i E_i \oplus \bigoplus_{k=0}^{l_i- 1} q_i^{2k} \mathrm{id} \buildrel \sim \over \longrightarrow E_iF_i .$$
\item If $l_i \le 0$, then
$$ q_i^{-2}F_i E_i  \buildrel \sim \over \longrightarrow  E_iF_i \oplus \bigoplus_{k=0}^{-l_i- 1} q_i^{-2k-2}  \mathrm{id} .$$
\end{enumerate}
\end{thm}

For $\nu \in \I^n$, let
\[
K_q(\lambda, \nu) := \; \sum_{\mathclap{\substack{\ttt \in \ST{\la}\\ \res{\ttt} = \nu}}} \; q^{\deg(\ttt)}, \qquad K_q(\lambda) := \; \sum_{\mathclap{\ttt \in \ST{\la}}} \; q^{\deg(\ttt)}.
\]

\begin{thm}[{\cite[Thm.2.5]{aps}}] \label{thm:dim}
For $\nu,\nu'\in I^\beta$, we have
\begin{align*}
\dim_q e(\nu) \fqH{}(\beta) e(\nu') &= \; \sum_{\mathclap{\substack{\lambda\vdash n\\ \wt(\lambda) = \Lambda - \beta}}}  \; K_q(\lambda, \nu) K_q(\lambda, \nu') ,\\
\dim_q \fqH{}(\beta) &= \; \sum_{\mathclap{\substack{\la \vdash n\\ \wt(\lambda) = \Lambda - \beta}}} \; K_q(\lambda)^2, \\
\dim_q \fqH{}_n &= \; \sum_{\la \vdash n} \; K_q(\lambda)^2,
\end{align*}
where $\dim_q M:= \sum_{k\in \bbz} \dim_\bbf(M_k)q^k$ for a free graded $\bbf$-module $M = \bigoplus_{k \in \bbz} M_k$.
\end{thm}

The statement below is an immediate consequence of the dimension formula:

\begin{cor}[\protect{cf.\ \cite[Cor.\ 2.7]{apc}}] 
\begin{enumerate}
\item Let $\nu \in I^n$. Then, $e(\nu) \ne 0 $ in $\fqH{k}(n)$ if and only if $\nu$ may be obtained from a standard tableau $\ttt$ as $\nu = \res{\ttt}$.
\item For a natural number $n$, we have $ \dim \fqH{k}_n = n!. $
\end{enumerate}
\end{cor}

\begin{prop}\label{prop:defect0}
    Block algebras of defect $0$ are simple.
\end{prop}

\begin{proof}
    By \cref{lem:defcalc,prop:maximal,reptype}, a block algebra $B$ of defect $0$ is derived equivalent to $ \fqH{k}_0 $ which is simple. 
    %Since the base field $\bbf$ is algebraically closed, $B$ and $ \fqH{k}_0 $ are isomorphic to matrix algebras and hence are local and symmetric. Chris:~we do not need this as we are not using a tilting argument here. 
    By \cite[Proposition~4.2]{APA2}, these algebras are self-injective and so by Rickard's theorem \cite[Theorem~2.1]{Rick89}, the derived equivalence between them is in fact a stable equivalence.
    Since $ \fqH{k}_0 $ is a simple algebra, all of its modules are projective and its stable module category consists only of zero objects, hence the same is true for $B$.
    
    Suppose that the unique indecomposable projective $B$-module is not simple. Then, the identity automorphism of the simple $B$-module does not factor through any projective $B$-module, so that it is a nonzero homomorphism in the stable module category of $B$, which is a contradiction. Hence, the unique simple $B$-module is projective-injective and therefore B is a simple algebra as well. 
\end{proof}

The next lemma will play a crucial role in our proofs.

\begin{lem}[{\cite[Lem.1.3]{arikirep}}] \label{trick}
Suppose that $e=e_1+e_2$ with $e_1^2=e_1\ne 0$, $e_2^2=e_2\ne 0$, $e_1e_2=e_2e_1=0$ and
$$
\dim_q e_i\fqH{k}(\beta)e_j-\delta_{ij}-c_{ij}q^2\in q^3\bbz_{\ge0}[q], 
$$
for $i,j=1,2$. Then the quiver of $e\fqH{k}(\beta)e$ has two vertices $1$ and $2$, $c_{ii}$ loops on the vertex $i$, for $i=1,2$, and there are at least $c_{12}$ arrows and 
$c_{21}$ reverse arrows between $1$ and $2$.
\end{lem}

\begin{lem}[{\cite[Lem.~3.1]{arikirep}}] \label{isom}
Let $\sigma: \I\simeq \I$ be a Dynkin automorphism, namely a bijective map that satisfies $a_{\sigma(i)\sigma(j)}=a_{ij}$, for $i,j\in I$. For 
$\beta=\sum_{i\in I} b_i\alpha_i\in \rlQ^+$ and $\Lambda=\sum_{i\in I} c_i\Lambda_i\in \wlP^+$, we define 
$$
\sigma\beta=\sum_{i\in I} b_i\alpha_{\sigma(i)}, \quad \sigma\Lambda=\sum_{i\in I} c_i\Lambda_{\sigma(i)}.
$$
Then, $R^\La(\beta)$ defined with $Q_{ij}(u,v)$ is isomorphic to $R^{\sigma\La}(\sigma\beta)$ defined with $Q'_{ij}(u,v)$ such that
$$
Q'_{\sigma(i)\sigma(j)}(u,v)=Q_{ij}(u,v).
$$
\end{lem}

\begin{cor}\label{cor:blockisom}
    We have an isomorphism of algebras
\[
    \fqH{k}(\beta)_{Q_{ij}}\cong  \fqH{\ell-k}(\sigma\beta)_{Q'_{ij}}
   \cong \fqH{\ell-k}(\sigma\beta)_{Q_{ij}}
\]
\end{cor}

\begin{proof}
    The first isomorphism follows by setting $\sigma(i)=\ell-i$ in \cref{isom}; the second isomorphism holds by the arguments given in the discussion in \cite[Lem.~2.2]{APA1}.
\end{proof}

Hence in order to determine the representation type of $\fqH{k}(\beta)$, it is enough to only consider $0\le k\le \ell/2$ by application of the above isomorphism.

\section{Representations of \texorpdfstring{$R^{\La_k}(\delta)$}{Rk(d)}}
 
In \cite{apc} in order to investigate the representation type of $\fqH{0}(\delta)$, they built on the work of \cite{km17}  and constructed the irreducible $\fqH{0}(\delta)$-modules.
However for $\fqH{k}(\delta)$, we only need to consider the projective modules and their radical series for $k=1$ and all other cases can be easily proved using \cref{trick}.
Note that $\fqH{k}(\delta)$ has defect $2$.
 
\begin{lem}\label{tame}
For $\ell=2$, the algebra $R^{\La_k}(\delta)$ has tame representation type.
\end{lem}

\begin{proof}
If $k=0$ this is proved in \cite[Thm.~3.7]{apc}. Suppose that $k\neq 0$. Then we may assume that $k=1$ by \cref{cor:blockisom} and we will prove the assertion by explicit construction of the indecomposable projective modules.
We have that $\delta=\al_0+2\al_1+\al_2$ and $e(\nu)\neq0\in \fqH{1}(\delta)$ for the following four sequences:
\begin{align*}
    e_1&=e(1210)\\
    e_2&=e(1201)\\
    e'_2&=e(1021)\\
    e_3&=e(1012).
\end{align*}
Using \cref{thm:dim}, we can easily calculate the $q$-dimensions of $e(\nu)\fqH{1}(\delta)\fqH{1}$:

\begin{center}
\begin{tabular}{ c|cccc } 
$e(\nu)\backslash e(\nu')$ & $1210$ & $1201$ & $1021$ & $1012$ \\  \hline
$1210$ & $1+q^4$ & $q^2$ & $q^2$ & $\cdot$ \\ 
$ 1201$ & $q^2$ & $ 1+q^2+q^4$ & $ 1+q^2+q^4$ & $q^2$ \\ 
$ 1021$ & $q^2$ &  $1+q^2+q^4$ & $1+q^2+q^4$ & $q^2$  \\ 
$ 1012$ & $\cdot$ & $q^2$ & $q^2$ & $1+q^4$ \\ 
\end{tabular}
\end{center}

Let $A=\fqH{1}(\delta)$.
Looking at the table above, we see that $A$ is non-negatively graded, hence its radical consists of linear combinations of elements of positive degree and $A/\Rad(A)$ is semisimple.
The basis of  $A/\Rad(A)$ contains the degree zero elements:
\[
A/\Rad(A)
=\spn\{e_1, e_2, e_2\psi_2e'_2, e'_2\psi_2e_2, e'_2, e_3\}.
\]

Let
\begin{align*}
    \calp_1&:=\fqH{1}(\delta)e_1\\
    \calp_2&:=\fqH{1}(\delta)e_2\\
    \calp_2'&:=\fqH{1}(\delta)e'_2\\
    \calp_3&:=\fqH{1}(\delta)e_3\\
\end{align*}
and define the simple modules $\cals_i:=\calp_i/ \Rad(\calp_i)$.
(Note that $\cals_1$ and $\cals_3$ are both one dimensional, while $\cals_2$ and $\cals_2'$ have dimension two.)

As for all $i=1,2,3,4$ we have that 
\[
\dim_q e_iAe_i=1+\text{higher order terms},
\]
the algebras $\End_A(\calp_i)$ are local, and we must have that $\calp_i$ are indecomposable projective $A$-modules with simple heads.

Now, we will show that $\calp_2\cong \calp_2'$ as left $A$-modules.
Let $f:\calp_2\rightarrow\calp_2'$ be given by $xe_2\mapsto x e_2\al$ where $\al=e_2\psi_2e_2'$. 
Because $f(e_2)\neq 0$ and it has degree $0$, we conclude that $f(e_2)\notin \Rad(\calp_2')$.
Thus, it generates $\cals_2'$ and $f$ is surjective.
Finally, as $\dim\calp_2=\dim\calp_2'$, we see that $f$ is indeed an isomorphism of $A$-modules.

Let $e=e_1+e_2+e_3$.
By looking at the $q$-dimensions of $e(\nu)Ae(\nu')$, we see that $\calp_i\not\cong\calp_j$ for $1\le i\neq j\le3$, hence $eAe$ is the basic algebra of $A$ and all simple $eAe$-modules are one dimensional.
We have the following radical series for $eAe$ (as the corresponding module categories of $eAe$ and $A$ are Morita equivalent, abusing notation, we write $\calp_i$ and $\cals_i$ instead of $e\calp_i$ and $e\cals_i$):
$$
\calp_1 \simeq
\begin{array}{c}
\cals_1 \\
\cals_2\\
 \cals_1
\end{array}
\qquad
\calp_2 \simeq
\begin{array}{c}
\cals_2 \\
\cals_1 \oplus \cals_2 \oplus \cals_3 \\
 \cals_2
\end{array}
\qquad
\calp_3 \simeq
\begin{array}{c}
\cals_3 \\
\cals_2\\
 \cals_3.
\end{array}
$$

Thus $eAe\cong\bbf Q/\cali$ where the ideal $\cali$ is given by the relations $\al_1\al_2=0=\beta_2\beta_1,\al_1\gamma=0=\gamma\beta_1,\gamma\al_2=0=\beta_2\gamma,\al_2\beta_2=\gamma^2=\beta_1\al_1$ and the quiver $Q$ has the following form:

\[
\begin{tikzcd}
1 \arrow[r, shift left, "\al_1"]
& 2 \arrow[l, shift left,"\beta_1"]
\arrow[r, shift left,"\al_2"]
\arrow[looseness=6, in=60, out=120, distance=0.8 cm,"\gamma"]
& 3 \arrow[l, shift left,"\beta_2"]
\end{tikzcd}
\]

Hence $\fqH{1}(\delta)$ is of tame representation type by \cite[Thm.~1(5)]{akmw20}.
\end{proof}

\begin{lem}\label{onedelta}
For $\ell\ge 3$ and $0\le k\le \ell$, the algebra $R^{\La_k}(\delta)$ has wild representation type.
\end{lem}

\begin{proof}
If $k=0$, the result follows by \cite[Thm.~3.7]{apc}.

Let $A=e\fqH{k}(\delta)e$ where $e=e_1+e_2$.
For $\ell=3$, we set:
\begin{align*}
    e_1&=e(123201)\\
    e_2&=e(101232).
\end{align*}

\begin{comment}
\cmt{I will just delete these tableaux. No need for them. Chris:~ok}
Then $e(123201)\neq 0$ for the following three tableaux: 
\begin{center}
\begin{tabular}{ m{2.5cm}m{2.5cm}m{2.5cm} } 
$\young(1234,5,6)$ & $\young(1234,56)$ &  $\young(12346,5)$ \\
$(4,1^2)$ & $(4,2)$ & $(5,1)$\\
$\deg 0$ & $\deg 1$ &  $\deg 2$  \\ 
\end{tabular}
\end{center}
and for $e(101232)\neq 0$ we have four different tableaux:
\begin{center}
\begin{tabular}{ m{2cm}m{2cm}m{2cm}m{2cm} } 
$\young(1,2,3,4,5,6)$ & $\young(16,2,3,4,5)$ &  $\young(145,2,3,6)$ & $\young(1456,2,3)$\\
$(1^6)$ & $(2,1^4)$ &  \new{$(3,1^3)$} & \new{$(4,1^2)$}\\
$ \deg 0$ & $\deg 1$ & $\deg 1$ & $\deg 2 $ \\ 
\end{tabular}
\end{center}
\end{comment}

Otherwise, we choose
\begin{align*}
    e_1&=e(k,k+1,\dots,\ell,\dots,k+1,k-1,k,k-2,k-2,\dots,0,\dots,k-1)\\
    e_2&=e(k,k+1,\dots,\ell,\dots,k+3,k+2,k-1,k-2,\dots,0,\dots,k+1).
\end{align*}

Using \cref{thm:dim}, we compute the graded dimensions for any $1\le k\le \ell/2$:
\begin{align*}
    \dim_qe_iR^{\La_k}(\delta)e_i&=1+c_{i,\ell}q^2+q^4;\\
    \dim_qe_iR^{\La_k}(\delta)e_j&=\dim_qe_jR^{\La_k}(\beta)e_i=q^2,
\end{align*}
where 
\begin{equation*}
c_{i,\ell}=
\begin{cases}
1 & \text{if}\qquad \ell=3,i=1\\
2 & \text{otherwise}.
\end{cases}
\end{equation*}

As $c_{i,\ell}>0$, we have that $e\fqH{k}(\delta)e_1$ and $e\fqH{k}(\delta)e_2$ are pairwise non-isomorphic, indecomposable projective $A$-modules and in particular that $R^{\La_k}(\delta)$ is of wild representation type for $\ell\ge 3$ by \cref{trick} and \cite[I.10.8(iv)]{er90}.
\end{proof}

\section{Representations of \texorpdfstring{$R^{\La_k}(\delta-\xik{k}{\pm 2})$}{Rk(d)}}

In this section we see that blocks of defect one have finite representation type, and moreover that they are equivalent to a Brauer tree algebra.
We also demonstrate that depending on $k$, there are two distinct possibilities for the number of simple modules in block algebras of defect one. Note that every block algebra of defect one is derived equivalent to $\fqH{k}(\delta-\xik{k}{\pm i})$ by \cref{lem:defcalc,prop:maximal,reptype}, 
%a block algebra of defect one is derived equivalent to $R^{\La_k}(m\delta - \xik{k}{\pm i})$ with $\df(m\delta - \xik{k}{\pm i}) = 2m - i/2 = 1 $ for $ m \ge i/2 $, and so we uniquely have that $ m = 1 $ and $ i = 2 $.
where
\begin{align*}
    \delta-\xik{k}{2}&=\al_0+2\al_1+\cdots+2\al_{k}+\al_{k+1},\text{ and} \\
    \delta-\xik{k}{-2}&=\al_{k-1}+2\al_{k}+\cdots+2\al_{\ell-1}+\al_\ell
\end{align*}
for $k\pm2 \in I$.

\begin{prop}\label{prop:xik2}
    $R^{\La_k}(\delta-\xik{k}{\pm 2})$ is of finite representation type.
\end{prop}

\begin{proof}
First consider the case for $\delta-\xik{k}{2}$.
The partitions belonging to the block algebra $R^{\La_k}(\delta-\xik{k}{2})$ are $ \lambda(i) = (2^{k-i+1},1^{2i}) $ for $0 \le i \le k+1$. 

The partition $\la(0) = (2^{k+1})$ is not Kleshchev since it has only one removable $1$-node, which is not a good node as it is immediately followed by an addable $1$-node (when reading from top to bottom), and hence does not show up in the reduced $1$-signature, and so $(2^{k+1})$ cannot correspond to a vertex in the crystal graph. 
The remaining partitions are all Kleshchev since each of them has a good (removable) node, and we may traverse against the directed edges on the crystal graph by continuing to remove good nodes until the empty partition $\varnothing_k$ is reached.
In more detail, starting from each partition we can move up the crystal graph as follows: the first column of $\la(i)$ has residue $(k,k-1,\ldots,1,0,1,\ldots i)$ and the second column has residue $ (k+1,k,\ldots,i+1) $ and so $\la(i)$ has a good $i$-node. The partition obtained by removing that good node has itself a good $(i-1)$-node, and the partition obtained from by removing that good node has a good $(i-2)$-node, and so on until the partition $(2^{k-i+1})$ is reached. Here, we have a unique removable node with residue $i+1$, and since $i>0$ that node is good and can be removed to move up the crystal graph. Continuing as before, the remainder of the second column, and subsequently the rest of the first column, can be removed by a sequence of good nodes until the empty partition $\varnothing_k$ is reached.
This is by no means the only possible path back to $\varnothing_k$, but to show that $\la(i)$ is Kleshchev for $ 1 \le i \le k+1 $ it suffices to exhibit one such path.

%These idempotents are sufficient: see Evseev-Mathas Lemma~5A.3 cf. Definition~4A.5 and $z_\la$
Let $ e_i $ be the idempotent corresponding to the residue sequence for the initial tableau for each Kleshchev partition $\la(i)$ for $0 \le i \le k$; in particular $e_i = e(\nu(i))$ where 
$$
 \nu(0) = (k, k-1,\ldots,1,0,1,\ldots,k,k+1)
$$
and for $1 \le i \le k$, the residue sequence $\nu(i)$ is
$$
 \nu(i) = (k, k+1, k-1, k, \ldots, k-i+1, k-i+2, k-i, k-i-1, \ldots, 1,0,1, \ldots, k-i+1)
$$

From \cref{thm:dim} above, we have the following graded dimensions:
$$
\dim_{q} e_i\fqH{k}(\delta-\xik{k}{2}) e_j=
\begin{cases} 1+{q}^2 \quad &\text{if $j=i$}, \\
{q} \quad &\text{if $j=i\pm 1$}, \\
 0 \quad &\text{otherwise}.
 \end{cases}
$$
Hence, $\Rad \fqH{k}(\delta-\xik{k}{2}) $ is spanned by homogeneous elements of positive degree, and $P_i=\fqH{k}(\delta-\xik{k}{2}) e_i$ are indecomposable and pairwise non-isomorphic projective modules and hence we can conclude that the radical series of $P_i$ is given by
$$
P_0=\begin{array}{c} S_0 \\ S_1 \\ S_0 \end{array}, \quad
P_i=\begin{array}{c} S_i \\ S_{i-1}\oplus S_{i+1} \\ S_i \end{array}\;\text{($1\le i\le k$)}, \quad
P_{k}=\begin{array}{c} S_k \\ S_{k-1} \\ S_k \end{array}
$$
where $S_i = \Top(P_i)$, and so $R^{\La_k}(\delta-\xik{k}{2})$ is Morita equivalent to a Brauer tree algebra whose Brauer tree is a straight line (cf. \cite[Prop.~5.1]{arikirep}). 
Moreover, the idempotents $e_1,e_2,\dots,e_k$ give a complete list of the pairwise non-isomorphic primitive idempotents of $\fqH{k}(\de-\xik{k}{2})$.

For the case $\delta-\xik{k}{-2}$, we can apply the same argument where the partitions belonging to $R^{\La_k}(\delta-\xik{k}{-2})$ are $\mu(i) = (\ell-k+i+1,\ell-k-i+1)$ for $0 \le i \le \ell-k+1$, all Kleshchev except when $ i = \ell-k+1 $
(by a similar reasoning to the previous case) and so in this case the idempotents correspond to the residue sequence $$ \upsilon(i) = (k, k+1, \ldots, \ell-1, \ell, \ell-1, \ldots, \ell-i, k-1, k, \ldots, \ell-i-1),  $$
for $0\le i \le \ell-k$ instead.
\end{proof}

\begin{cor}
    When $ k \neq \ell/2 $ and $ 2 \le k \le \ell-2 $, there are two inequivalent Morita equivalence classes of blocks of defect one with distinct number of simple modules $k+1$ or $\ell-k+1$ respectively.
\end{cor}

\begin{proof}
    From the proof of \cref{prop:xik2}, we see that $R^{\La_k}(\delta-\xik{k}{-2})$ has $k+1$ simple modules and $R^{\La_k}(\delta-\xik{k}{+2})$ has $\ell-k+1$ simple modules.
    By \cref{lem:defcalc,prop:maximal,reptype} all blocks of defect one have the same number of simple modules as either one of these two maximal weight cases.
\end{proof}
In \cite{aps}, the authors construct for cyclotomic KLR algebras of type $C^{(1)}_\ell$ the Specht modules $S^\la$ where $\la$ is a (multi)partition. The following corollary describes the graded decomposition matrices for blocks of defect one in level one, $ \La = \La_k $.
\begin{cor}\label{cor:defect1}
    Let $m \in \{ k+1, \ell-k+1 \}$ be the number of simple modules for a block algebra of defect one. Then the graded decomposition multiplicity is

    $ [S^{\la(i)}:D^{\la(j)}]_q =
    \begin{cases}
          1 &\text{ if }  i=j,   \\
          q &\text{ if }  i=j+1 \text{ for } 1 \le i \le m+1, \\
          0 &\text{otherwise, } \\
    \end{cases} \\
    $
    where $\la(i)$ is as in the proof of \cref{prop:xik2} above when $m = k+1$, and $\la(i)=\mu(i)$ when $ m = \ell-k+1 $.
\end{cor}

\begin{proof}
    From the proof of \cref{prop:xik2}, the ungraded decomposition multiplicities are $[S^{\la(i)}:D^{\la(j)}] = 1$ when $i=j$ or $i=j+1$ or $1 \le i \le m+1$. By Corollary~6E.18 of \cite{EM22} when the defect $d$ is 1, the bottom-most nonzero entry of each column of the graded decomposition matrix is $ [S^{\la(i)}:D^{\la(i-1)}]_q = q^{d} = q$ for $1\le i \le m+1$.
\end{proof}

We note here that $R^{\La_k}(\delta-\xik{k}{\pm 2})$ is an indecomposable algebra.

\begin{thm}\label{thm:xik2}
    Let $B$ be a block algebra of defect $1$. Then $B$ is of finite representation type and moreover $B$ is a Brauer tree algebra whose Brauer tree is a straight line with no exceptional vertex.
\end{thm}

\begin{proof}
    By \cref{lem:defcalc,prop:maximal,reptype,prop:xik2}, $B$ is of finite type and is derived equivalent to $R^{\La_k}(\delta-\xik{k}{\pm2})$, which is a Brauer line algebra.
    Moreover, derived equivalence preserves Hochschild homology so in particular, the zeroth Hochschild homology i.e. the center is preserved and hence $B$ is indecomposable as well.
    Furthermore, $B$ is (graded) cellular by \cite[Thm.~A]{EM22} and symmetric by \cite[\S4E]{EM22}.
    Hence, applying Ohmatsu's theorem and by the argument in \cite[\S8.2]{arikirep}, we conclude that $B$ is a Brauer tree algebra whose tree is a straight line with no exceptional vertex.
\end{proof}

\begin{rem}
   Thus we see that the algebra $\fqH{k}(\beta)$ is indecomposable if it has defect $0$ or $1$.
   We expect this to be true for any $\fqH{}(\beta)$ in any defect as well.
\end{rem}

\section{Representations of \texorpdfstring{$R^{\La_k}(2\delta-\xik{k}{\pm 4})$}{Rk(2d-xk)}}

Note that every block algebra of defect $2$ is derived equivalent to $R^{\La_k}(\delta)$ or $R^{\La_k}(2\delta-\xik{k}{\pm 4})$.
Observe that 
\begin{align*}
    \beta:=2\delta-\xik{k}{4}&=2\al_0+4\sum^k_1\al_j+3\al_{k+1}+2\al_{k+2}+\al_{k+3},\,\text{and}\\
     \gamma:=2\delta-\xik{k}{-4}&=\al_{k-3}+2\al_{k-2}+3\al_{k-1}+4\sum_k^{\ell-1}\al_j+2\al_\ell.
\end{align*}

\begin{lem}\label{xik4}
For $k\pm 4 \in I$, the algebra $R^{\La_k}(2\delta-\xik{k}{\pm 4})$ has wild representation type.
\end{lem}

\begin{proof}
If $k=0$, this is proved in \cite[Thm.~4.2]{apc}.
Suppose $k\neq0$ and $k\le \ell/2$.
Then, $k+4\in I$ implies that $\ell > 4$ and $k-4 \in I $ implies that $\ell \ge 2k \ge 8$ and so we may assume that $\ell > 4$.
Just as in the proof of \cref{onedelta}, we set $e=e_1+e_2$.
First, we consider the algebra $A=eR^{\La_k}(\beta)e$ where we choose
\begin{align*}
        e_1=&(k,k-1,\dots,0,\dots,k+3,k+1,k+2,k,\dots,0,\dots,k+1)\\
        e_2=&(k,k+1,k+2,k+3,k-1,k-2,\dots,0,\dots,k+2,\\
        &\hspace{1cm}k,k+1,k-1,k-2,\dots,0,\dots,k).
\end{align*}

For shorthand notation, let $\underline{i}:=(i,i-1,i-2,i-3)$.
Let $B=eR^{\La_k}(\gamma)e$ and set

\begin{align*}
    e_1=&(\underline{k},\underline{k+1},\dots,\underline{\ell},\ell-1,\ell-2,\ell,\ell-1)\\
    e_2=&(\underline{k},\underline{k+1},\dots,\underline{k+d-2},\ell-1,\ell,\ell-1,\ell-2,\ell-3,\\
    &\hspace{1cm}\ell-4,\ell-2,\ell-3,\ell-1,\ell,\ell-1,\ell-2).
\end{align*}

\begin{comment}
\cmt{I will delete the degrees.}
For $e_1$, we have the following shapes and
corresponding degrees: 
\begin{itemize}
    \item $ \deg 0:(p+3,p+2,p+2,p+1)$,
    \item $\deg 1:(p+3,p+3,p+1,p+1)$
    \item $\deg 1:(p+4,p+2,p+2,p) $
    \item $\deg 2:(p+4,p+3,p+1,p)$;
\end{itemize}

and for $e_2$: 
\begin{itemize}
    \item $ \deg 0:(p+4,p+3,p+1,p)$,
    \item $\deg 1:(p+4,p+4,p,p)$
    \item $\deg 1:(p+5,p+3,p+1,p-1) $
    \item $\deg 2:(p+5,p+4,p,p-1)$.
\end{itemize}
\end{comment}

Then using \cref{thm:dim}, we can compute the graded dimensions for $\kappa=\beta,\gamma$ and any $k>0$:
\begin{align*}
    \dim_qe_iR^{\La_k}(\kappa)e_i&=1+2q^2+q^4;\\
    \dim_qe_iR^{\La_k}(\kappa)e_j&=\dim_qe_jR^{\La_k}(\beta)e_i=q^2.
\end{align*}

Thus $R^{\La_k}(\kappa)$ is wild by the same reasoning as in \cref{onedelta}.
\end{proof}

\section{Representations of \texorpdfstring{$R^{\La_k}(\beta)$}{Rk(b)}}

In this final section, we prove two important lemmas that will enable us to generalise our results for all $R^{\La_k}(\beta)$.
After stating the main theorem, we also rewrite it in terms of defect and compare it with the original statement of Erdmann--Nakano.

\begin{lem}[{\cite[Prop.~2.3]{EN02}}, {\cite[Rem.~5.10]{apd}}]
\label{functor}
Let $A$ and $B$ be finite dimensional $\bbf$-algebras and suppose that
there exists a constant $C>0$ and functors
$$
F:\;A\text{\rm -mod} \rightarrow B\text{\rm -mod}, \quad
G:\;B\text{\rm -mod} \rightarrow A\text{\rm -mod}
$$
such that, for any $A$-module $M$,
\begin{itemize}
\item[(1)]
$M$ is a direct summand of $GF(M)$ as an $A$-module,
\item[(2)]
$\dim F(M)\le C\dim M$.
\end{itemize}
Then, if $A$ is wild, so is $B$.
\end{lem}

The next lemma is the analogous statement of \cite[Lem.~5.3]{aps} for $\La_k$.

\begin{lem} \label{derived}
\begin{enumerate}
\item [(1)] If $\fqH{k}(\beta - \alpha_j)$ is wild and $\langle h_j, \La_k - \beta + \alpha_j \rangle \ge 1$, then $\fqH{k}(\beta)$ is wild.
\item[(2)] Suppose that $\fqH{k}(n \delta - \xik{k}{\pm i})$ is wild. Then
\begin{enumerate}
\item[(a)] $\fqH{k}((n+1) \delta - \xik{k}{\pm i})$ is wild,
\item[(b)] if $k\pm(i+2) \in I$, then $\fqH{k}((n+1) \delta - \xik{k}{\pm (i+2)})$ is wild.
\end{enumerate}
\end{enumerate}
\end{lem}

\begin{proof}
(1) Considering the functors
\[
F_j: \fqH{k}(\beta - \alpha_j)\text{-mod} \rightarrow \fqH{k}(\beta)\text{-mod}, \quad E_j: \fqH{k}(\beta)\text{-mod} \rightarrow \fqH{k}(\beta- \alpha_j)\text{-mod},
\]
the statement follows from \cref{functor,endo}.

(2) First, we will consider $\fqH{k}((n+1) \delta - \xik{k}{i})$.
Notice that 
\begin{align*}
\Lambda_k + \xik{k}{i+2}-(n+1)\delta&=\La_k+\xik{k}{i}+\al_{k+i+1}+2\sum^{\ell-1}_{j=k+i+2}\al_j+\al_\ell-(n+1)\delta\\
&=\La_k+\xik{k}{i}-\al_0-2\sum^{k+i}_{j=1}\al_j-\al_{k+i+1}-n\delta.
\end{align*}
For $0\le k+i \le \ell-1$ and $n\in \bbz_{\ge0}$, we compute
\begin{align*}
\Lambda_k + \xik{k}{i+2}- (n+1)\delta + \alpha_{k+i+1} &= r_{k+i}r_{k+i-1}\dots r_{1} r_{0} r_{1} \dots r_{k+i} (\La_k + \xik{k}{i} - n\delta),\\
\La_k + \xik{k}{i} - (n+1)\delta + \alpha_\ell &= r_{\ell-1}r_{\ell-2}\dots r_{1} r_{0} r_{1} \dots r_{k+i} (\La_k + \xik{k}{i} - n\delta).
\end{align*}
Hence $\fqH{k}((n+1)\delta - \xik{k}{i+2} - \alpha_{k+i+1})$ and $\fqH{k}((n+1)\delta - \xik{k}{i} - \alpha_{\ell})$ are wild by \cref{reptype} and the assumption that $\fqH{k}(n\delta-\xik{k}{\pm i})$ is. Moreover, we also have that
\begin{align*}
&\langle h_{k+i+1}, \Lambda_k + \xik{k}{i+2} - (n+1)\delta + \alpha_{k+i+1} \rangle = 2,\\
&\langle h_{\ell}, \Lambda_k + \xik{k}{i} - (n+1)\delta + \alpha_\ell \rangle = 2,
\end{align*}
and now we apply (1) to arrive at the desired conclusion.

Similarly, for $k+i=\ell$ we need to consider
\[
\Lambda_k + \xik{k}{i}- (n+1)\delta + \alpha_0 = r_{1}r_{2} \cdots r_{\ell} (\Lambda_k + \xik{k}{i} - n\delta). 
\]
and we also have to check
\[
\langle h_0, \Lambda_k + \xik{k}{i} - (n+1)\delta + \alpha_0 \rangle = 2.
\]
Using \cref{reptype} and $(1)$ again, $(2)$ follows for $\fqH{k}(n\delta - \xik{k}{i})$.

Next, we look at $\fqH{k}((n+1) \delta - \xik{k}{-i})$. Here we note that
\begin{align*}
\Lambda_k + \xik{k}{-(i+2)}-(n+1)\delta&=\La_k+\xik{k}{-i}+\al_{0}+2\sum^{k-(i+2)}_{j=1}\al_j+\al_{k-(i+1)}-(n+1)\delta\\
&=\La_k+\xik{k}{-i}-\al_{k-(i+1)}-2\sum^{\ell-1}_{j=k-(i+2)}\al_j-\al_\ell-n\delta.
\end{align*}
The proof is essentially the same as for $k+(i+2)$, but in this case the Weyl group generators $r_i$ will act in reverse order.
For $1\le k-i \le \ell$ and $n\in \bbz_{\ge0}$, we compute
\begin{align*}
\Lambda_k + \xik{k}{-(i+2)} - (n+1)\delta + \alpha_{k-(i-1)} &= r_{k-i}r_{k-(i-1)}\dots r_{\ell-1} r_{\ell} r_{\ell-1} \cdots r_{k-i} (\Lambda_k + \xik{k}{-i} - n\delta),\\
\Lambda_k + \xik{k}{-i} - (n+1)\delta + \alpha_0 &= r_{1}r_{2}\dots r_{\ell-1} r_{\ell} r_{\ell-1} \dots r_{k-i} (\Lambda_k + \xik{k}{-i} - n\delta).
\end{align*}
Moreover, we also have that
\begin{align*}
&\langle h_{k-(i+1)}, \Lambda_k + \xik{k}{-(i+2)}- (n+1)\delta + \alpha_{k-(i-1)} \rangle = 2,\\
&\langle h_{0}, \Lambda_k + \xik{k}{-i} - (n+1)\delta + \alpha_0 \rangle = 2.
\end{align*}
Finally, for $k-i=0$ we need to consider
\[
\Lambda_k + \xik{k}{-k} - (n+1)\delta + \alpha_\ell = r_{\ell-1}r_{\ell-2} \cdots r_{1}r_0 (\Lambda_k + \xik{k}{-k}- n\delta). 
\]
and we also have to check
\[
\langle h_\ell, \Lambda_k + \xik{k}{-i} - (n+1)\delta + \alpha_\ell \rangle = 2.
\]
Thus for $\fqH{k}(n\delta - \xik{k}{-i})$ $(2)$ follows by the same reasoning as for $k+(i+2)$.
\end{proof}

\begin{thm}\label{twodelta}
The algebra $R^{\La_k}(2\delta)$ is wild.
\end{thm}

\begin{proof}
If $\ell\ge 3$, the statement follows from \cref{onedelta} by applying \cref{derived} with $i = 0$. Now assume $\ell=2$. Let $e=e_1+e_2$ and consider $A=eR^{\La_1}(2\delta)e$ where 
\begin{equation*}
        e_1=e(10121012) \hspace{1cm}\text{and}\hspace{1cm}
        e_2=e(12012101).
\end{equation*}

Then we have the following graded dimensions:
\begin{align*}
    \dim_qe_1R^{\La_1}(2\delta)e_1&=1+2q^2+2q^4+2q^6+q^8\\
    \dim_qe_2R^{\La_1}(2\delta)e_2&=1+2q^2+3q^4+2q^6+q^8 \\
    \dim_qe_1R^{\La_1}(2\delta)e_2&=\dim_qe_2R^{\La_1}(2\delta)e_1=q^2+2q^4+q^6.
\end{align*}

Then $R^{\La_1}(2\delta)$ is wild by the same reasoning as in \cref{onedelta}.
\end{proof}

\begin{thm}\label{2xik2}
The algebra $R^{\La_k}(2\delta-\xik{k}{\pm 2})$ is wild.
\end{thm}

\begin{proof}
If $k=0$, this is proved in \cite[Lem.~5.4]{apc}.
Suppose $k\neq 0$ and $k\le \ell/2$.
Then $k+2\le \ell$ implies that $\ell \ge 3$ and $k-2 \in I$ implies that $\ell\ge 2k \ge 4$ and so we may assume that $\ell \ge 3$. 
Notice that for $1\le k \le \ell-1$ we have that 
\begin{align*}
    2\delta-\xik{k}{2}&=\delta+\al_0+2\sum^k_{j=1}\al_j+\al_{k+1}\\
    2\delta-\xik{k}{-2}&=\delta+\al_{k-1}+2\sum^{\ell-1}_{j=k}\al_j+\al_{\ell}.
\end{align*}
By \cref{onedelta}, we already know that $\fqH{k}(\delta)$ is wild and we also see that
\begin{align*}
&\langle h_k, \La_k - \delta \rangle,
\langle h_{k+1}, \La_k - \delta-\al_k \rangle,
\dots,
\langle h_{\ell-1}, \La_k - \delta-\al_k-\dots-\al_{\ell-2} \rangle,\\
&\langle h_\ell, \La_k - \delta-\al_k-\dots-\al_{\ell-1} \rangle,
\langle h_{k-1}, \La_k - \delta-\al_k-\dots-\al_{\ell} \rangle
\end{align*}
are all positive, so we have that $\fqH{k}(\delta+\al_{k-1}+\al_k+\dots+\al_\ell)$ is also wild by \cref{derived}.
Finally, direct computation shows that
\begin{align*}
\La_k-(\delta+\al_{k-1}+2\sum^{\ell-1}_{j=k}\al_j+\al_{\ell})&=
    \La_k-\delta-\al_{k-1}-2\al_k-2\al_{k+1}-\dots-2\al_{\ell-1}-\al_\ell\\
    &=r_{\ell-1}r_{\ell-2}\dots r_k(\La_k-\delta-\al_{k-1}-\al_k-\dots-\al_\ell).
\end{align*}
Hence $R^{\La_k}(2\delta-\xik{k}{-2})$ is wild.

Similarly, for $2\delta-\xik{k}{2}$ it is easy to see that
\begin{align*}
&\langle h_k, \La_k - \delta \rangle,
\langle h_{k-1}, \La_k - \delta-\al_k \rangle,
\dots,
\langle h_1, \La_k - \delta-\al_2-\dots-\al_k \rangle,\\
&\langle h_0, \La_k - \delta-\al_1-\dots-\al_k \rangle,
\langle h_{k+1}, \La_k - \delta-\al_0-\dots-\al_{k} \rangle
\end{align*}
are all positive, thus the algebra $\fqH{k}(\delta+\al_0+\al_1+\dots+\al_k)$ is wild.
Using direct computation again, we have that
\begin{align*}
(\La_k-\delta-\al_{0}-\al_1-2\al_2-\dots-2\al_{k}-\al_{k+1})\\
    &=r_2r_3\dots r_k(\La_k-\delta-\al_0-\al_1-\dots-\al_{k+1})
\end{align*}
and
that
\begin{equation*}
\langle h_1, \La_k-\delta-\al_0-\al_1+2\sum^{k}_{j=2}\al_j+\al_{k+1}\rangle=2
\end{equation*}
thus $R^{\La_k}(2\delta-\xik{k}{2})$ is wild by \cref{derived} and we have proved the statement.
\end{proof}

\begin{thm}\label{main}
Let $0\le k+i \le \ell$ or $0\le k-i \le \ell$ for some even $i \in I$.
% Changed the minus sign (see APC errata) 
For $\beta \in \weyl(\Lambda_k + \xik{k}{\pm i})$ and $m\ge i/2$, the block algebra $R^{\La_k}( \La_k - \beta +  m\delta )$ of type $C_\ell^{(1)}$ is
\begin{enumerate}
\item a simple algebra if $i=m=0$;
\item of finite representation type if $m=1$ and $i=2$;
\item of tame representation type if $i=0$, $m=1$ and $\ell=2$; and
\item of wild representation type otherwise.
\end{enumerate}
\end{thm}

\begin{proof}
If $k=0$ or $\ell$, the result follows by \cite{apc}.
If $1\le k \le \ell-1$, $(1)$ comes from \cref{prop:defect0}, $(2)$ comes from \cref{thm:xik2}, $(3)$ is proved in \cref{tame} and $(4)$ follows by \cref{onedelta} and applying \cref{derived} to \cref{xik4,twodelta,2xik2}.
\end{proof}

\begin{rem}
We note here that we can phrase things in terms of defect. Let $d$ be the defect of $R^{\La_k}( \La_k - \beta +  m\delta )$ for $m\ge i/2$, $\beta \in \weyl(\Lambda_k + \xik{k}{\pm i})$ with even $i \in I$ such that $0\le k+i \le \ell$ or $0\le k-i \le \ell$. Then the block algebra $R^{\La_k}( \La_k - \beta +  m\delta )$ of type $C_\ell^{(1)}$ is
\begin{enumerate}
\item a simple algebra if $d=0$;
\item of finite representation type if $d=1$;
\item of tame representation type if $\ell=2$ and $d=2$; and
\item of wild representation type otherwise.
\end{enumerate}
\end{rem}

\begin{rem}
    We also summarise the results of Ariki--Ijima--Park, Ariki--Park \cite{APA1,APA2,apc,apd} and ours on the representation type of block algebras in level one in terms of defect.
    Let $d$ denote the defect of the respective block algebra.
    \begin{itemize}
        \item {If $\La=\La_0$} in type $A^{(2)}_{2\ell}$,
$d=0$ implies simple, $d=1$ implies finite, wild otherwise (tame representation type does not occur here).
\item If $\La=\La_0$ in type $D^{(2)}_{\ell+1}$,
$d=0$ implies simple, $d=1$ implies finite, $d=2$ implies tame, wild otherwise.
\item If $\La=\La_k$ in type $A^{(1)}_\ell$,
$d=0$ implies simple, $d=1$ implies finite, $d=2$ and $\ell=1$ implies tame, wild otherwise. (This follows from the fact that in type $A^{(1)}_\ell$, $\fqH{0}(\beta)\cong\fqH{k}(\beta)$ for any $0\le k\le \ell$.)
\item If $\La=\La_k$ in type $C^{(1)}_\ell$, $d=0$ implies simple, $d=1$ implies finite, $d=2$ and $\ell=2$ implies tame, wild otherwise.
    \end{itemize}
\end{rem}

\bibliographystyle{amsalpha}
\addcontentsline{toc}{section}{\refname}
\bibliography{master}

\end{document}